\documentclass[12pt]{amsart}
\usepackage{amssymb,amsmath}

\newcommand{\R}{{\mathbb R}}
\newcommand{\Rn}{{\mathbb{R}^3}}
\newcommand{\RN}{{\mathbb{R}^{3}}}
\newcommand{\be}{\begin{equation}}
\newcommand{\ee}{\end{equation}}
\newcommand{\n }{\nabla }
\newcommand{\od}{\omega^{2}}
\newcommand{\diff}{{\setminus}}

\newcommand{\pa}{\partial}

\newcommand{\grad}{\nabla}
\newcommand{\irn}{\int_{\Rn}}
\newcommand{\hsob}{H^{1}}
\newcommand{\huno}{H^1}
\newcommand{\hh}{H^1\times H^1}
\renewcommand{\H}{{\mathbb H}}

\renewcommand{\a }{\alpha}

\renewcommand{\b }{\beta }

\newcommand{\vep}{\varepsilon}
\newcommand{\eps}{\varepsilon}
\newcommand{\vfi}{\varphi}
\newcommand{\ovphi}{\bar{\phi}}
\newcommand{\ovpsi}{\bar{\psi}}
\newcommand{\ovu}{\overline{u}}
\newcommand{\ovv}{\overline{v}}
\newcommand{\g }{\gamma }
\newcommand{\ga }{\gamma }

\renewcommand{\t}{\theta}

\newcommand{\Ga}{\Gamma}

\renewcommand{\S}{\Sigma}

\newcommand{\om}{\omega}

\newcommand{\Ne}{{\mathcal N}}
\renewcommand{\SS}{{\mathbb S}}

\numberwithin{equation}{section}
\newtheorem{theorem}{Theorem}[section]
\newtheorem{proposition}[theorem]{Proposition}

\newtheorem{lemma}[theorem]{Lemma}
\newtheorem{definition}[theorem]{Definition}
\theoremstyle{definition}
\newtheorem{remark}[theorem]{Remark}
\newtheorem{remarks}[theorem]{Remarks}

\renewcommand{\dfrac}{\displaystyle\frac}
\newcommand{\brm}{\begin{remark}\rm}
\newcommand{\erm}{\end{remark}}
\newcommand{\bte}{\begin{theorem}}
\newcommand{\ete}{\end{theorem}}
\newcommand{\bpr}{\begin{proposition}}
\newcommand{\epr}{\end{proposition}}
\newcommand{\ble}{\begin{lemma}}
\newcommand{\ele}{\end{lemma}}
\newcommand{\beq}{\begin{equation}}
\newcommand{\eeq}{\end{equation}}
\newcommand{\bdm}{\begin{displaymath}}
\newcommand{\edm}{\end{displaymath}}
\numberwithin{equation}{section}
\newcommand{\rife}{\eqref}

\title[Semiclassical states for CNLS systems]
{Semiclassical states for weakly coupled \\ nonlinear Schr\"odinger systems}

\author[E.\ Montefusco]{Eugenio Montefusco}
\author[B.\ Pellacci]{Benedetta Pellacci}
\author[M.\ Squassina]{Marco Squassina}

\address{Dipartimento di Matematica
\newline\indent
Universit\`a degli Studi di Roma ``La Sapienza''
\newline\indent
P.le A.\ Moro 5, I-00185 Roma, Italy}
\email{montefusco@mat.uniroma1.it}

\address{Dipartimento di Scienze Applicate
\newline\indent
Universit\`a degli Studi di Napoli ``Parthenope''
\newline\indent
Via De Gasperi 5, I-80133 Napoli, Italy}
\email{benedetta.pellacci@uniparthenope.it}

\address{Dipartimento di Matematica e Applicazioni
\newline\indent
Universit\`a  di Milano Bicocca
\newline\indent
Via R.\ Cozzi 53, I-20125 Milano, Italy}
\email{marco.squassina@unimib.it}

\thanks{The first and the second author are supported by the MIUR national
research project ``Variational Methods and Nonlinear Differential Equations'',
while the third author is supported by the MIUR national research
project ``Variational and Topological Methods in the Study of
Nonlinear Phenomena'' and by the Istituto Nazionale di Alta
Matematica (INdAM)}


\begin{document}
\begin{abstract}
We consider systems of weakly coupled Schr\"odinger equations with nonconstant
potentials and we investigate the existence of nontrivial nonnegative solutions
which concentrate around local minima of the potentials. We obtain
sufficient and necessary conditions  for a sequence of least energy
solutions  to concentrate.
\end{abstract}
\maketitle


\section{Introduction}
Starting from the celebrated works \cite{berplions,coffman,strauss},
the recent years have been marked out by
an ever-growing interest in the study of standing wave solutions to
the semi-linear Schr\"odinger equation (NLS)
\begin{equation*}
i\phi_t+\Delta\phi+|\phi|^2\phi=0 \qquad\text{in $\R^3\times(0,\infty)$},
\end{equation*}
where i denotes the imaginary unit.
As a related problem, a large amount of work
(see \cite{abc,ams,ambook,dpf,fw,wangzeng}
and references therein) has been devoted to
the study of the semiclassical states for (NLS), namely the study
of the singularly perturbed equation
$-\eps^2\Delta u + V(x)u=u^3$ in $\R^3$ for $\eps$ going to zero,
where $V(x)$ is a potential modeling the action of external forces.
Under different hypotheses on the potential $V$ it has been proved
that there exists a family of solutions $\{u_\eps\}$ which exhibits a
spike shape around the non-degenerate critical points of $V$ and
decays elsewhere.
\\
From a physical point of view,  the nonlinear Schr\"odinger
equation arises  in the study of  nonlinear optics in isotropic materials,
for instance the pro\-pagation of pulses in a {\em single-mode} nonlinear optical fiber.
However, a single-mode optical fiber is actually {\em bi-modal} due to the
presence of some birefringence effects which tend to split a pulse into
two pulses in two different polarization directions.
Menyuk \cite{men} showed that, under various simplifications and
variable scalings, the complex amplitudes of the two wave packets
$\phi$ and $\psi$ in a birefringence optical fiber are governed by a
system of two coupled nonlinear Schr\"odinger equations ((CNLS) for short).
Looking for standing wave solutions leads to study the following
elliptic system
\beq
\label{prob0}
\begin{cases}
-\Delta u+u=u^3+bv^{2}u & \text{in $\R^3$} , \\
-\Delta v+\od v=v^3+bu^{2}v & \text{in $\R^3$},
\end{cases}
\eeq
where $b$ is a real-valued cross phase coefficient depending upon the
ani\-so\-tro\-py of the fiber, and $\omega$ is the frequencies ratio of the two
waves. Physically, $b>0$ is known as the attractive case, whereas $b<0$ is the
repulsive case. Apart from some special cases, the study of \eqref{prob0} is
pretty complicated.
This because of the presence of {\sl semitrivial} or {\sl scalar}
solutions, indeed, there always exist the solutions $(u,0),\, (0,v)$ with
$u,\,v$ solutions of the single equations in  \eqref{prob0}; then it becomes
physically relevant to know whether or not  a solution found  is really vectorial,
i.e. with both nontrivial components. Recently, this problem has been tackled
in  \cite{ac0,ac,mmp} by means of different methods. In particular, in
\cite{mmp} it has been proved that for $b$ sufficiently small every ground
state solution necessarily has one trivial component, while for $b$ sufficiently
large the ground state solutions have both positive components.
As far as concern the semiclassical states, we are naturally lead to study
the system
\begin{equation}
\label{problema}
\tag{$S_\eps$}
\begin{cases}
-\eps^2\Delta u + V(x)u =u^3+bv^2u & \text{in $\R^3$}, \\
\noalign{\vskip2pt}
-\eps^2\Delta v + W(x)v =v^3+bu^2v & \text{in $\R^3$}.
\end{cases}
\end{equation}
This is the goal of this paper.  We will assume that the potentials
$V,W$ are H\"older continuous functions in $\R^{3}$, bounded from
below away from zero and $\eps$ is a small parameter which will approach
zero. Our intent is to show the existence, for small $\eps$, of a
nonnegative (i.e. with nonnegative components) least energy solution
$(u_\eps,v_\eps)$ and  then  to prove sufficient and necessary conditions concerned
with the concentration of  $(u_\eps,v_\eps)$  around  the local minimum
(possibly degenerate) points of the potentials, which are supposed to be in
 the same region.
Aiming to use variational methods,  we will consider the functional
$J_\eps$ associated to \eqref{problema}, which satisfies all the assumptions of the
Mountain Pass theorem (\cite{ambrab}) except for the Palais-Smale condition
since we do not assume any global condition on $V,\,W$.
Then, we will use a vectorial adaptation of the argument in \cite{dpf}; namely
we will perform a pe\-na\-li\-za\-tion of $J_\eps$, exploiting the homogeneity
of the nonlinearities,  outside the region containing the minimum points  of
the potentials, so that we will consider a modified  functional  which
satisfies all the hypotheses of the Mountain Pass theorem
including the Palais-Smale condition. To show the concentration, we will argue on
the sum $u_\eps(x)+v_\eps(x)$ proving that it is uniformly, with
respect to $\eps$,  bounded away from zero, and by taking advantage of the known
properties of the autonomous system we can show that $u_\eps(x)+v_\eps(x)$
has exactly one global maximum point, which tends to a minimum point of $V$ or
$W$. Here we cannot be more precise without assuming some conditions on $b$
as one between $u_\eps$ and $v_\eps$ may vanish or not as $\eps\to 0$.
Namely, we can show that for $b$ smaller than a positive constant $b_0$
(defined in \eqref{b-def}) either $u_\eps$ or $v_\eps$ necessarily expires and
the other tends--up to scalings--to the least energy solution of the
corresponding autonomous nonlinear Schr\"odinger equation.
When $b$ is large (greater than a positive constant $b_1$ defined in \eqref{b-def})
both $u_\eps$ and $v_\eps$ survive and we recover
a least energy vectorial solution of the autonomous system (see Theorem \ref{concentrazione}).
As physically reasonable, for materials with low anisotropy, one component of the
system is predominant upon the other, since the low birefringence is not able
to split a soliton-type solution in two distinct pulses. Recently,
it was studied in \cite{pomp} the repulsive case $b<0$ (a model for the
Bose-Einstein condensation). We stress that the methods used therein are
very different from ours, since the change of sign of the constant $b$ produces a different
behavior of the solutions (see also \cite{fp} for the case of a single equation).
\vskip2pt
Concerning the necessary conditions for a sequence of solutions to concentrate,
contrary to the scalar case with power nonlinearity (\cite{abc,ams}), we cannot in general
derive an explicit representation of the so called ground energy function $\Sigma$ (see
formulas \eqref{defnehari}-\eqref{defsigma}).
The underlying philosophy is that when the limit problem \eqref{prob0} lacks
of uniqueness, then the ground energy function, which will be shown to be at least locally
Lipschitz continuous, may lose its additional smoothness properties.
Nevertheless, in this framework, on the line of \cite{secsqu}, we prove
that a necessary condition for a family of solutions $(u_\eps,v_\eps)$ to
concentrate around a given point $z$,  is that $z$ is a critical point, not
necessarily a minimum point,  of $\Sigma$ in the sense of the Clarke
subdifferential $\partial_C$, that is $0\in \partial_C\Sigma(z)$.
Moreover, due to the previously mentioned characterization of least energy
solutions in terms of the coupling parameter $b$ (see proposition \ref{lili}),
we partition the concentration points $\mathcal{E}$ into three classes
$\mathcal{E}=\mathcal{E}_V\cup\mathcal{E}_W\cup\mathcal{E}_\Sigma,$
where
$$
\mathcal{E}_V\times \mathcal{E}_W\times \mathcal{E}_\Sigma
\subset {\rm Crit}(V)\times {\rm Crit}(W)\times {\rm Crit}_C(\Sigma).
$$
denoting ${\rm Crit}(f)$ (resp.\ ${\rm Crit}_C(f)$) the set of
classical (resp.\ in the sense of the Clarke subdifferential)
critical points of a function $f$. In this partition we can see again that if a family of
solution concentrates around a given point then we derive as a limit
problem either a single equation or the entire system, depending on the value of $b$.
Namely, we will find some positive constants $b^\infty_0<b^\infty_1<b^\infty_2$
such that for  $b<b^\infty_0$ we obtain the single equation as the limit problem,
for $b>b^\infty_2$ we show that if a family of least energy solutions concentrates, then
its scaling around a minimum point of the potentials converges to a real vectorial
least energy solution of the autonomous system.
\vskip8pt
The plan of the paper is the following.
\vskip2pt
\noindent
In Section \ref{statem} we introduce the functional setting and the statements
of the main results.
In Section \ref{suff-ps} we proceed with the proof of the
main result regarding the sufficient conditions for concentration.
In Section \ref{nec-ps} we prove the main achievement on the necessary
conditions for the concentration.

\section{The functional framework and main statements}
\label{statem}

Let $V(x)$ and $W(x)$ be H\"older continuous functions in $\R^{3}$ and
suppose that there exists a positive constant $\a$ such that
\beq\label{positiveVW}
V(x),\,W(x)\geq\a \qquad\text{for all $x\in\Rn$}.
\eeq
In order to study $(S_{\eps})$ we use variational methods, so that we
introduce the Hilbert space
\bdm
\H=\Big\{(u,v)\in\huno\times\huno:\,\,\int_{\R^{3}}V(x)u^{2}<\infty,\,\,\
\int_{\R^{3}}W(x)v^{2}<\infty \Big\},
\edm
where $H^1=H^1(\R^3)$ is the usual first order Sobolev space in $\R^3$.
The norm in $\H$ is  $\|(u,v)\|_{\H}^2=\|u\|_{\eps,V}^{2}+\|v\|_{\eps,W}^{2}$,
where
$$
\|u\|_{\eps,V}^{2}=\eps^{2}\|\grad u\|_{2}^{2}+\int_{\Rn}V(x)u^{2},\qquad
\|v\|_{\eps,W}^{2}=\eps^{2}\|\grad v\|_{2}^{2}+\int_{\Rn} W(x)v^{2},
$$
being $\eps$ a small parameter and where we denote with $\|\cdot\|_p$
the standard norm in $L^p=L^p(\Rn)$ for $1\leq p\leq\infty$.
We will study the functional $J_{\eps}:\H\to\R$ defined by
\bdm
J_{\eps}(u,v)=\dfrac{1}{2}\|u\|_{\eps,V}^{2}+\dfrac{1}{2}\|v\|_{\eps,W}^{2}
-\int_{\R^{3}}F(u,v),
\edm
where we have set
$$
F(u,v)=\dfrac{1}{4}(u^{4}+2bu^2v^{2}+v^{4}),\qquad\text{with $b>0$}.
$$
It is easily  checked that $J_{\eps}$ is well defined and of class $C^{1}$ on $\H$.
A nontrivial solution of problem $(S_\vep)$ is a couple $(u_\vep,v_\vep)\neq (0,0)$
in $\H$, critical point of $J_{\eps}$.
\\
We denote by $B(x,r)$ the open ball centered at $x$ with radius $r$ and
with $\partial B(x,r)$ its boundary.

As far as concern the sufficient conditions for the concentration to occur, we will
prove two main results; the first is the following.

\begin{theorem}\label{concentrazione}
Assume \rife{positiveVW} and that there exist $z\in\Rn$ and $r>0$ such that
\begin{align}
\label{minimoV}
V_0&=\min_{B(z,r)}V<\min_{\pa B(z,r)}V, \\
\noalign{\vskip2pt}
\label{minimoW}
W_0&=\min_{B(z,r)}W<\min_{\pa B(z,r)}W.
\end{align}
Then there exists $\eps_{0}>0$ such that, for every $0<\eps<\eps_{0}$,
problem $(S_{\eps})$ admits a nontrivial solution $(u_{\eps},v_{\eps})\in\H$,
$u_{\eps},\,v_{\eps}\geq0$, such that the following facts hold:
\vskip8pt
\noindent
{\rm (i)} $(u_{\eps}+v_{\eps})$ admits exactly one global maximum
point $x_{\eps}\in B(z,r)$ with
\be\label{conveVW}
\lim_{\eps\to 0} V(x_{\eps})=V_0
\quad
\text{ or }
\quad
\lim_{\eps\to 0} W(x_{\eps})=W_0.
\ee
Furthermore, there exist $\mu_1,\mu_2>0$ such that, for every $x\in\R^3$,
$$
u_\eps(x)+v_\eps(x)\leq \mu_1e^{-\mu_2\textstyle\frac{|x-x_\eps|}{\eps}}.
$$
\vskip8pt
\noindent{\rm (ii)} Let us define $b_0<b_1$ by
\begin{equation}\label{b-def}
b_0= \max \left\{\sqrt[4]{{\dfrac{W_0}{V_0}}},
\sqrt[4]{{\dfrac{V_0}{W_0}}}\right\},
\qquad
b_1=\max\left\{h\left(\sqrt{{\dfrac{W_0}{V_0}}}\right),
h\left(\sqrt{{\dfrac{V_0}{W_0}}}\right)\right\},
\end{equation}
 with
\begin{equation}\label{defh1}
h(s) =\min \left\{\frac{s}{32}\left(7+\frac1{s^2}\right)^2-1, \frac{s^2+3}4\right\}.
\end{equation}
\vskip4pt
\noindent
Then the following facts hold:
\vskip4pt
\noindent
- if $b<b_0$, there exists $\sigma>0$ such that  for all $0<\eps<\eps_0$,
$$
\text{either $u_\eps(x_\eps)\to 0$ and $v_\eps(x_\eps)\geq\sigma$ or $v_\eps(x_\eps)\to 0$
and $u_\eps(x_\eps)\geq\sigma$}.
$$
\vskip2pt
\noindent
- if $b>b_1$, then there exist $\sigma>0$  such that, for all $0<\eps<\eps_0$,
\begin{equation*}
u_\eps(x_\eps)\geq\sigma,\qquad v_\eps(x_\eps)\geq\sigma.
\end{equation*}
\end{theorem}

\vskip2pt
\begin{remarks}
\label{d}
\begin{enumerate}
\item Actually, we can be more precise in conclusion (ii) of Theorem \ref{concentrazione}.
Indeed, if $V_0<W_0$ $u_\eps$ converges to zero while $v_\eps(x_\eps)$
remains bounded away from zero; otherwise if $W_0<V_0$ $u_\eps$ survives
and $v_\eps$ expires (see Remark \ref{survivor} for more details).
\item In the case $V=W$, there holds $b_0=b_1=1$;
then for $b<1$ $(u_\vep,v_\vep)$ converges (up to scalings) to the least
energy solution of one of the equations in $(S_\vep)$. While, for $b>1$
 $(u_\vep,v_\vep)$ converges to a real vector solution of the system
$(S_\vep)$.
\item The constants $b_{0}$ and $b_1$ depend only on the minima
$V_{0}, W_{0}$, so that $V$ and $W$ may have a degenerate minimum point
or a closed, connected bounded set of nonnegative measure of minimum points;
\item When considering the action of external forces in the propagation of
pulses in optical fibers, the potentials in the model problem are $V(x)$ and
$W(x)=V(x)+c$ with $c$ positive constant. In this case the result follows just
by assuming that \eqref{minimoV} holds.
\end{enumerate}
\end{remarks}
We can also prove a more general result than Theorem \ref{concentrazione}.
In order to do this, let us define the functional with frozen potentials $I_{z}:H^1\times H^1\to\R$
\be\label{defiz}
I_{z}(u,v)=\frac12\|u\|_{z}^2+\frac12\|v\|^{2}_{z}-\int_{\R^3} F(u,v),
\ee
where  $\|u\|^{2}_{z}=\|\nabla u\|^{2}_{2}+V(z)\|u\|^{2}_{2}$  for every $u\in
\hsob$. The critical points of $I_z$ are the solutions of  the system
\begin{equation}
\label{limit-z}
\tag{$S_z$}
\begin{cases}
-\Delta u + V(z)u = u^3+bv^2u & \text{in $\R^3$}, \\
\noalign{\vskip2pt}
-\Delta v + W(z)v = v^3+bu^2v & \text{in $\R^3$}.
\end{cases}
\end{equation}
The Nehari manifold  associated to $I_{z}$ is defined by
\be\label{defnehari}
{\mathcal N}_{z}=\left\{(u,v)\in \hsob\times\hsob\setminus\{(0,0)\}\,:\,\langle
I'_{z}(u,v),(u,v)\rangle=0\right\},
\ee
and the infimum of $I_z$ on ${\mathcal N}_{z}$
\be\label{defsigma}
\Sigma(z)=\inf_{{\mathcal N}_{z}}I_{z}\,.
\ee
Following the same argument of Lemma 3.1 in \cite{mmp}, it is possible to prove
that the Mountain Pass level of $I_{z}$ is equal to $\Sigma(z)$. In the
following we will denote with $(\varphi_{z},\psi_{z})\neq(0,0)$ the point where
$I_{z}$ achieves $\Sigma(z)$, that is $(\varphi_{z},\psi_{z})$ will be a least
energy solution of \rife{limit-z} (see \cite{bl} or \cite{mmp}, for example).

Because of this property, the function $\Sigma$ is known as the ground energy
function and plays an important role when studying necessary and sufficient
conditions for the concentration to occur, as the following result shows.

\begin{theorem}\label{concentrazioneG}
Assume \eqref{positiveVW} and that there exist $z\in\Rn$ and $r>0$ such that
\begin{align}
\label{minimoS}
\Sigma_0=\min_{B(z,r)}\Sigma<\min_{\pa B(z,r)}\Sigma.
\end{align}
Then there exists $\eps_{0}>0$ such that, for every $0<\eps<\eps_{0}$,
problem $(S_{\eps})$ admits a nontrivial solution $(u_{\eps},v_{\eps})\in\H$,
$u_{\eps},\,v_{\eps}\geq0$, such that $(u_{\eps}+v_{\eps})$ admits exactly
one global maximum point $x_{\eps}\in B(z,r)$ with
\be\label{conveS}
\lim_{\eps\to 0} \Sigma(x_{\eps})=\Sigma_0,
\ee
and conclusions (i) and (ii) of Theorem \ref{concentrazione} hold true.
\end{theorem}

\begin{remark}
Theorem \ref{concentrazioneG} is more general than
Theorem \ref{concentrazione}. Indeed, conditions \eqref{minimoV}-\eqref{minimoW}
imply the desired information \eqref{minimoS} (see for the details the proof in Section
\ref{concentrazione}). However, Theorem \ref{concentrazioneG} is an abstract
result since we cannot write down explicitly the function $\Sigma$, due to the possible
lack of uniqueness of least energy solutions of $(\eqref{limit-z})$. It would
be interesting to see if, by assuming the $\Sigma$ admits a 'topologically
nontrivial' Clarke critical point, the concentration still pops up.
\end{remark}

Aiming to state a necessary condition for a family of solutions $(u_\eps,v_\eps)$
to concentrate around a  point $z$, we need a few preliminary definitions.

\begin{definition}
Let $z\in\R^3$ and let $b_{z}\geq 1$ be defined by
\beq\label{bzzeta}
b_z={\textstyle\max\left\{\sqrt[4]{\dfrac{W(z)}{V(z)}},
\sqrt[4]{\dfrac{V(z)}{W(z)}}\right\}}.
\eeq
For every $b>0$, we put
$$
\mathcal{O}_b=\left\{z\in\R^3:\,
b_{z}\geq b\right\}.
$$
\end{definition}

Next we define the concentration sets.

\begin{definition}\label{defe}
The concentration set for system~\eqref{problema}, $\mathcal{E}$, is defined by
\begin{align*}
\mathcal{E}=\Big\{& z\in\R^3:\, \text{there exists a sequence of
solutions $(u_{\eps},v_{\eps})\in\H$ of \eqref{problema} with} \\
& \text{$u_{\eps}(z+\eps x)+v_{\eps}(z+\eps x)\to 0$ as $|x|\to\infty$ uniformly with respect} \\
& \text{to $\eps$ and ${\eps}^{-3}J_{\eps}(u_{\eps},v_{\eps})\to\Sigma(z)$
as $\eps\to 0$} \Big\}.
\end{align*}
We  also introduce the subsets of $\mathcal{E}$
\begin{align*}
\mathcal{E}_V&:=\big\{ z\in\mathcal{E}\cap \mathcal{O}_b:\, \text{$u_{\eps}(z)\geq\delta$
for some $\delta>0$ and any $\eps>0$}\big\},    \\
\noalign{\vskip2pt}
\mathcal{E}_W&:=\big\{ z\in\mathcal{E}\cap \mathcal{O}_b:\, \text{$v_{\eps}(z)\geq\delta$
for some $\delta>0$ and any $\eps>0$}\big\},     \\
\noalign{\vskip4pt}
\mathcal{E}_\Sigma&:=\mathcal{E}\setminus\mathcal{O}_b.
\end{align*}
\end{definition}
In general the function $\Sigma$ is not known to be differentiable, but it is
always locally Lipschitz, as we will see. On the other hand, we need to consider
the critical points of $\Sigma$, so that we will use the
Clarke subdifferential (see \cite{clarke}), which is well defined for a
locally Lipschitz function. We will need the following definition.

\begin{definition}
For $V,\,W\in C^1(\R^3)$ and $\Sigma\in {\rm Lip}_{\rm loc}(\R^3)$ we denote by ${\rm Crit}(V)$ and ${\rm Crit}(W)$ the sets of the
critical points in $\mathcal{O}_b$ of $V$ and $W$ respectively, and
by ${\rm Crit}_C(\Sigma)$ the set of  $z\not\in \mathcal{O}_b$  critical points
of $\Sigma$  in the sense of Clarke subdifferential, that is:
\begin{align*}
{\rm Crit}(V)&=\big\{z\in\mathcal{O}_b:\,\nabla V(z)=0\big\}, \\
\noalign{\vskip2pt}
{\rm Crit}(W)&=\big\{z\in\mathcal{O}_b:\,\nabla W(z)=0\big\}, \\
\noalign{\vskip2pt}
{\rm Crit}_C(\Sigma)&=\big\{z\not\in\mathcal{O}_b:\,\partial_C\Sigma(z)\ni 0\big\},
\end{align*}
where
\begin{equation*}
\partial_C \Sigma(z)=\big\{\eta\in\R^3:\,\,
\Sigma^0(z,w)\geq \eta\cdot w,\,\,\,\text{for every $w\in\R^3$}\big\},
\end{equation*}
being $\Sigma^0(z,w)$ the generalized derivative of
$\Sigma$ at $z$ along $w\in\R^3$, defined by
\begin{equation*}
\Sigma^0(z;w)=\limsup_{\substack{\xi \to z
\\ \lambda \to 0^{+}}}\frac{\Sigma(\xi+\lambda w)-\Sigma(\xi)}{\lambda}.
\end{equation*}
\end{definition}

We can now state the following necessary condition.

\begin{theorem}
\label{necessthm}
Assume \rife{positiveVW} and that $V,\,W\in C^1(\R^3)$ with
\begin{equation}
\label{growthVW}
|\nabla V(x)|\leq \beta e^{\gamma|x|}
\qquad
\text{and}
\qquad
|\nabla W(x)|\leq \beta e^{\gamma|x|},
\end{equation}\
for all $x\in\R^3$ and for some
constants $\beta>0$ and $\gamma\geq 0$.
Then $\Sigma$ is locally Lipschitz continuous and the following facts hold:
\vskip8pt
\noindent
{\rm (a)} $\mathcal{E}_V\cap\mathcal{E}_W\cap\{z\in\Rn:V(z)\neq W(z)\}=\emptyset$ and
$$
\mathcal{E}=\mathcal{E}_V\cup\mathcal{E}_W\cup\mathcal{E}_\Sigma,
$$
where
$$
\mathcal{E}_V\times \mathcal{E}_W\times \mathcal{E}_\Sigma
\subset {\rm Crit}(V)\times {\rm Crit}(W)\times {\rm Crit}_C(\Sigma).
$$
\vskip8pt
\noindent
{\rm (b)} If  $V,W\in L^\infty$, let $b_0^\infty<b_1^\infty<b^\infty_2$
be defined by
\begin{align}
\label{bstar1}
b_0^\infty &=
\max\Bigg\{\sqrt[4]{\dfrac{\alpha}{\|V\|_{\infty}}},
\sqrt[4]{\dfrac{\alpha}{\|W\|_\infty}}\Bigg\}, \\
\noalign{\vskip4pt}
\label{bstar2}
b_1^\infty &=\max\left\{\sqrt[4]{{\dfrac{\|V\|_\infty}{\alpha}}},
\sqrt[4]{{\dfrac{\|W\|_\infty}{\alpha}}}\right\}, \\
\noalign{\vskip2pt}
\label{bstar3}
b_2^\infty &=\max\left\{h\left(\sqrt{{\dfrac{\|V\|_\infty}{\alpha}}}\right),
h\left(\sqrt{{\dfrac{\|W\|_\infty}{\alpha}}}\right)\right\},
\end{align}
where $h$ is defined in \eqref{defh1}. Then
\begin{equation*}
\mathcal{E}=
\begin{cases}
\mathcal{E}_V\cup\mathcal{E}_W  & \text{for all $b\leq b_0^\infty$}, \\
\noalign{\vskip2pt}
\mathcal{E}_\Sigma & \text{for all $b>b_1^\infty$}.
\end{cases}
\end{equation*}
In addition, for every $b>b_2^\infty$ both the components of the solution
remain bounded away from zero from below.
\end{theorem}

\begin{remark}
\label{sottodiff}
As $\mathcal{E}_\Sigma\subset {\rm Crit}_C(\Sigma)$, in particular,
for $z\in\mathcal{E}_\Sigma$, it holds
$$
0\in {\rm Co}\Big\{\lim_{j\to\infty}\nabla\Sigma(\xi_j):
\,\,\text{$\xi_j\not\in{\mathcal D}$ and $\xi_j\to z$}\Big\},
$$
where ${\rm Co}$ denotes the convex hull and ${\mathcal D}$ is any
null set containing the set of points at which $\Sigma$ fails to be
differentiable. This follows by a well known
property of the Clarke subdifferential (see e.g.\ \cite{clarke}).
\end{remark}
\begin{remark}
Assume for a moment that system \eqref{limit-z} admits a unique ground state solution, up
to translations. Then, in light of formulas \eqref{form-dder} it follows
that $\Sigma$ is differentiable at $z$, $\partial_C\Sigma(z)=\{\nabla\Sigma(z)\}$
and hence, $\nabla\Sigma(z)=0$ provided that $z\in\mathcal{E}_\Sigma$. On the
other hand, we point out that, in general, \eqref{limit-z}  lacks of uniqueness
of ground state solutions. For instance, if $b=V(z)=W(z)=1$ and $U$ is
the unique solution to $-\Delta U+U=U^{3}$ in $\Rn$,
then the pairs $(\cos(\t) U,\sin(\t) U)$ with $0\leq\t\leq\pi/2$
are all ground states solutions. In the case $b<1$, by the results of
\cite{mmp} the system has at least the scalar least energy
solutions $(0,U)$ and $(U,0)$. In the case $b>1$, we suspect
that the system admits a unique ground state solution.
On the other hand, up to now, a proof seems out of reach.
\end{remark}


\section{Proof of Theorem \ref{concentrazione}}
\label{suff-ps}

We will follow the arguments used in \cite{dpf} for the single equation.
Let $\gamma>0$ be such that
\beq
\label{kg}
\gamma<\dfrac{\alpha}{3\sqrt{\max\{1,b\}}}.
\eeq
For any $s,t\in\R$, let us set
\bdm
F_\sharp(s,t)=\left\{\begin{array}{ll}
\displaystyle\frac14\left(s^{4}+2bs^2t^{2}+t^{4}\right)
& \text{if $s^{4}+2bs^2t^{2}+t^{4}\leq\g^{2}$},
\\ \\
\displaystyle\frac\gamma2\sqrt{s^{4}+2bs^2t^{2}+t^{4}}-\frac{\gamma^2}{4}
& \text{if $s^{4}+2bs^2t^{2}+t^{4}\geq\g^{2}$};
\end{array}\right.
\edm
it follows that
\bdm
\grad F_\sharp(s,t)=\left\{\begin{array}{ll}
\left((s^{2}+bt^{2})s,(t^{2}+bs^{2})t\right)\quad
& \text{if $s^{4}+2bs^2t^{2}+t^{4}\leq\g^{2}$},
\\ \\
\gamma\dfrac{\left((s^{2}+bt^{2})s,(t^{2}+bs^{2})t\right)}
{\sqrt{s^{4}+2bs^2t^{2}+t^{4}}}\quad & \text{if
$s^{4}+2bs^2t^{2}+t^{4}\geq\g^{2}$}.
\end{array}\right.
\edm
It is easy to see that $F_\sharp\in C^{1}(\R^2)$.
Let $B(z,r)$ a ball of radius $r$ centered in $z$ with $z$ satisfying
conditions \eqref{minimoV}-\eqref{minimoW}; we define
\bdm
G(x,s,t)=\chi(x) F(s,t)+(1-\chi(x))F_\sharp(s,t),
\edm
for a.e. $x\in\R^3$ and any $s,t\in\R$, where $\chi$ is the
characteristic function of the ball $B(z,r)$.
In the light of the above definition, it follows that the following conditions
hold for every $(s,t)$ in $\R^2$
\begin{equation}\label{g11}
0\leq 3G(x,s,t)<\grad G(x,s,t)\cdot(s,t)\quad \forall\,x\in B(z,r),\,
\end{equation}
and, for every $x\not\in\overline{B(z,r)}$,
\begin{equation}\label{g12}
0\leq 2G(x,s,t)\leq \grad G(x,s,t)\cdot(s,t)\leq \dfrac{1}{k}\left[
V(x)s^{2}+W(x)t^{2} \right]
\quad \text{with }k>3.
\eeq
We study the following functional
\bdm
\tilde{J}_{\eps}(u,v)=\frac12\|u\|_{\eps,V}^{2}
+\frac12\|v\|_{\eps,W}^{2}-\int_{\R^3}G(x,u,v).
\edm
Note that $\tilde{J}_{\eps} $ is of class $C^1$ on $\H$ and its critical
points solve the system
\begin{equation}
\label{eqg}
\begin{cases}
-\eps^2\Delta u + V(x)u = G_u(x,u,v) & \text{in $\R^3$}, \\
\noalign{\vskip2pt}
-\eps^2\Delta v + W(x)v = G_v(x,u,v) & \text{in $\R^3$}.
\end{cases}
\end{equation}
For each $\eps>0$ fixed, we will find a critical point of $\tilde{J}_{\eps}$
by applying the Mountain Pass theorem (\cite{ambrab}), so that we define
\begin{equation}\label{defc}
c_\eps=\inf_{\ga\in\Ga}\sup_{t\in[0,1]}\tilde{J}_{\eps}(\ga(t)),
\end{equation}
where $\Ga=\{\ga\in C([0,1],\H):\ga(0)=(0,0),\,\, \tilde{J}_{\eps}(\ga(1))<0\}$.
Arguing as in Lemma 2.1 in  \cite{dpf} and as in Lemma 3.2 in \cite{mmp}
it is possible to prove that
\beq\label{livelli}
c_\eps=\inf_{(u,v)\in \H\setminus \{(0,0)\}}\sup_{t\geq0}\tilde{J}_{\eps}(tu,tv).
\eeq
Moreover, we will compare $c_\eps$ with the level $\Sigma(z)$ (defined in \eqref{defsigma})
of a ground state solution $(\vfi_{z},\psi_{z})$ of the limit system \rife{limit-z}.
It is well known (see e.g. \cite{busi}, \cite{mmp}) that the functions $\vfi_{z},\psi_{z}$
are radially symmetric, nonnegative functions which decay exponentially to zero
at infinity.

First of all, we show that $\tilde{J}_{\eps}$ possesses suitably
estimated critical values.

\ble\label{exx}
Assume $\eqref{positiveVW}$. Then $\tilde{J}_{\eps}$ has a nontrivial critical point
$(u_{\eps},v_{\eps})\in\H$ such that
\beq\label{energia}
\tilde{J}_{\eps}(u_{\eps},v_{\eps})\leq \eps^{3} (\Sigma(z)+o(1)),
\eeq
where $o(1)\to 0$ as $\eps\to 0$.
Moreover, there exists a positive
constant $c_0$ such that
\beq
\label{stimato}
\|u_\eps\|_{\eps,V}^{2}+\|v_\eps\|_{\eps,W}^{2}\leq c_0\eps^{3}.
\eeq
\ele

\begin{proof}
Note that $(0,0)$ is a local minimum of the functional $\tilde{J}_{\eps}$, since
it holds $\tilde{J}_{\eps}(u,v)\geq c\|(u,v)\|^{2}_{\H}$,
provided that the norm $\|(u,v)\|_{\H}$ is sufficiently small. Moreover,
let $(\phi,\psi)\in \H$ with ${\rm supp}(\phi)\cup {\rm supp}(\psi)\subset B(z,r)$
and observe that $\tilde{J}_{\eps}(t(\phi,\psi))\to-\infty$ as $t\to+\infty$.
Then we can construct a Palais-Smale sequence at level $c_{\eps}$ (defined in
\eqref{defc}). Conditions \eqref{g11} and \eqref{g12} imply that
hypothesis $(g3)$ in \cite{dpf} is satisfied in our context, so  that the
compactness of  Palais-Smale sequences can be recovered following the
proof of Lemma 1.1 in \cite{dpf}. By applying the Mountain Pass Theorem
(\cite{ambrab}),  we get a nontrivial critical point
$(u_\eps,v_\eps)$ at level $c_\eps$.
In order to show estimate \rife{energia}, we need to consider a suitable
pair of functions which models the concentration phenomenon.
Let us define the functions
\bdm
u^{*}(x)=\eta(x)\vfi_{z}\left(\frac{x-z}{\eps}\right)\qquad
v^{*}(x)=\eta(x)\psi_{z}\left(\frac{x-z}{\eps}\right),
\edm
where $\eta$ is a smooth function compactly supported in $B(z,r)$
and such that $\eta=1$ in a small neighborhood of $z$ and
$(\vfi_{z},\psi_{z})$ is a ground state solution of problem \eqref{limit-z}.
From the definitions of $G(x,s,t)$ and $\eta(x)$ we deduce that
$\tilde{J}_\eps(tu^*,tv^*)=J_\eps(tu^*,tv^*)$, so that it is easy to compute
the supremum for $t\geq 0$ of  $\tilde{J}_{\eps}(tu^*,tv^*)$ and by using
\eqref{livelli} we derive
\bdm
\tilde{J}_{\eps}(u_{\eps},v_{\eps})=c_\eps\leq\sup_{t\geq0}
\tilde{J}_{\eps}(t u^{*},t v^{*}) =\eps^{3}\left[\Sigma(z)+o(1)\right],
\edm
that is \eqref{energia} holds.
Finally, using \eqref{energia}, the weak form of \rife{eqg} tested with
$(u_{\eps},v_{\eps})$ and \eqref{g11}, \eqref{g12}, it is possible to get
also \eqref{stimato}.
\end{proof}

In the next proposition the asymptotic behavior outside $B(z,r)$
of the critical point  $(u_\eps,v_\eps)$ found in Lemma \ref{exx}
is studied.

\begin{proposition}\label{maximo}
Assume $\eqref{positiveVW}$ and that $z\in\R^3$ and $r>0$ satisfy conditions
$\eqref{minimoV}$ and $\eqref{minimoW}$. Then for every $\delta>0$
there exists $\eps_\delta>0$ such that
\begin{equation}\label{limepsfin}
\sup_{0<\eps<\eps_\delta}\,\sup_{x\in\Rn\setminus B(z,r)}(u_{\eps}(x)+v_{\eps}(x))<\delta.
\end{equation}
\end{proposition}

\begin{proof}
Let us first prove that
\begin{equation}\label{limeps}
\lim_{\eps\to 0}\,\sup_{x\in\partial B(z,r)}(u_{\eps}(x)+v_{\eps}(x))=0.
\end{equation}
We proceed by contradiction, assuming that there exist
a sequence $\{\eps_{n}\}$ con\-ver\-ging to $0$ and a sequence $\{x_{n}\}\subset\pa B(z,r)$
such that, for some positive constant $\beta$,
\beq\label{contro}
u_{\eps_n}(x_{n})+v_{\eps_n}(x_{n}) \geq\b\qquad\text{for all $n\geq 1$.}
\eeq
Since $\pa B(z,r)$ is a compact set, we can assume that there exists a
subsequence of $\{x_{n}\}$, still denoted by $\{x_{n}\}$, which converges to a
point $x_{0}\in\pa B(z,r)$. Consider the scalings of $u_{\eps_{n}}$ and $v_{\eps_{n}}$
centered at $x_n$, that is
$$
\phi_{n}(x)=u_{\eps_{n}}(x_{n}+\eps_{n}x)\qquad
\psi_{n}(x)=v_{\eps_{n}}(x_{n}+\eps_{n}x),
$$
which are critical points of the functional $J_n$ defined in $\H$ by
\bdm\label{defJn}
\tilde J_{n}(u,v)=\frac{1}{2}\|u\|_{1,V(x_{n}+\eps_{n}x)}^{2}+
\frac{1}{2}\|v\|_{1,W(x_{n}+\eps_{n}x)}^{2}-\int_{\Rn}G(x_{n}+\eps_{n}x,u,v),
\edm
so that the couple $(\phi_n,\psi_n)$ solve the system
\beq\label{succ}
\begin{cases}
-\Delta\phi_{n}+V(x_{n}+\eps_{n}x)\phi_{n}=G_{u}(x_{n}+\eps_{n}x,\phi_{n},\psi_{n}), &
\\ \noalign{\vskip3pt}
-\Delta\psi_{n}+W(x_{n}+\eps_{n}x)\psi_{n}=G_{v}(x_{n}+\eps_{n}x,\phi_{n},\psi_{n}).
\end{cases}
\eeq
Notice that, by a simple change of scale, it is possible to verify that
\begin{equation}\label{change}
\tilde J_{n}(\phi_{n},\psi_{n})=\eps_{n}^{-3}\tilde{J}_{\eps_{n}}
(u_{\eps_{n}},v_{\eps_{n}}).
\end{equation}
From \eqref{stimato} we have that the sequences $\phi_n$ and $\psi_n$ are
bounded in $H^1$; this, \eqref{succ} and elliptic regularity estimates
imply that $\phi_n$ and $\psi_n$ converge $C^2$ on compact sets to a
couple $(\phi,\psi)\in\H$, which, by \eqref{contro} must be nontrivial.
In addition, there exists a function $\xi\in L^\infty$, with $0\leq \xi\leq1$ such
that $\chi(x_{n}+\eps_{n}x)$ converges to $\xi$ weakly* in $L^{\infty}$.
Then, the pair $(\phi,\psi)$ is a solution of
\bdm
\begin{cases}
-\Delta\phi+V(x_{0})\phi=\widehat G_{u}(x,\phi,\psi), &
\\
\noalign{\vskip3pt}
-\Delta\psi+W(x_{0})\psi=\widehat G_{v}(x,\phi,\psi), &
\end{cases}
\edm
where $\widehat G(x,s,t)=\xi(x) F(s,t)+(1-\xi(x))F_\sharp(s,t).$
The preceding system is the Euler equation of the functional
\bdm
J_{x_0}(u,v)=\frac{1}{2}\|u\|_{1,V(x_{0})}^{2}
+\frac{1}{2}\|v\|_{1,W(x_{0})}^{2}-\int_{\Rn} \widehat G(x,u,v).
\edm
On the other hand, conditions \eqref{g11} and \eqref{g12} allow us
to follow the same arguments of Lemma 2.2 in \cite{dpf}
to deduce that
\beq\label{claim}
\liminf_{n\to \infty} J_{n}(\phi_{n},\psi_{n}) \geq J_{x_0}(\phi,\psi).
\eeq
Indeed, consider the function
\begin{align*}
h_{n}=& \frac12\left[|\nabla \phi_{n}|^{2}+|\nabla\psi_{n}|^{2}
+V(x_{n}+\eps_{n}x)|\phi_{n}|^{2}+W(x_{n}+\eps_{n}x)|\psi_{n}|^{2}\right] \\
&- G(x_{n}+\eps_{n}x,\phi_{n},\psi_{n}).
\end{align*}
Choosing $R>0$ sufficiently large, from the $C^{1}$ convergence of
$\phi_{n},\,\psi_{n}$ over compacts, and since $\phi$ and $\psi$ belong
to $\hsob$ we have, for every $\delta>0$ fixed,
$$
\lim_{n\to\infty}\int_{B_{R}}h_{n}\geq J_{x_{0}}(\phi,\psi)-\delta,
$$
where $B_R$ stands for $B(0,R)$.
Moreover, taking $\eta_{R}$  a smooth cut-off function such that $\eta_{R}=
0$ on $B_{R-1}$ and $\eta_{R}= 1$ on $\RN\setminus B_{R}$, and using as
test function in \eqref{succ} $w=\eta_{R}(\phi_{n},\psi_{n})$, it is possible
to obtain
$$
\liminf_{n\to\infty}\int_{\RN\setminus B_{R}}h_{n}\geq -\delta,
$$
yielding \eqref{claim}.
Since $(\phi,\psi)$ is a critical point
of $J_{x_0}$ we have
\begin{equation}\label{maxiJ}
J_{x_0}(\phi,\psi)=\max_{t\geq0}J_{x_{0}}(t(\phi,\psi)).
\end{equation}
Moreover, it holds $F(s,t)\geq F_\sharp(s,t)$, so that $\widehat{G}(x,s,t)\leq F(s,t)$
which, together with \eqref{maxiJ}, implies that
\begin{equation}\label{diseq}
J_{x_0}(\phi,\psi)\geq \inf_{(u,v)\in \H}\sup_{t\geq0} I_{x_{0}}
(t(u,v))=\Sigma(x_{0}).
\end{equation}
From assumptions \rife{minimoV}, \rife{minimoW} it follows that
$V(x_{0})>V_0$ and $W(x_{0})>W_0$,
this means that $\Sigma(x_0)>\Sigma(z)$, where $\Sigma(z)$ is defined in
\eqref{defsigma}. This, \eqref{change}, \eqref{claim} and \eqref{diseq} yield
\begin{equation}
\label{cruc-step}
\Sigma(z)<J_{x_0}(\phi,\psi)\leq  \liminf_{n\to\infty}J_{n}(\phi_{n},\psi_{n})\leq \Sigma(z),
\end{equation}
which is a contradiction, proving \eqref{limeps}.

We are now ready to conclude the proof of the result. Let us fix $\delta>0$;
from \eqref{limeps} it follows that there exists $\eps_\delta>0$ such that
$0\leq u_{\eps}(x)<\delta$ and $0\leq v_{\eps}(x)<\delta$
for any $x\in\pa B(z,r)$ and $\eps\in(0,\eps_\delta)$.
It follows that $(u_{\eps}-\delta)^{+}= 0$ and $(v_{\eps}-\delta)^{+}= 0$ on
$\pa B(z,r)$ and hence we can choose
$$
\phi_\eps=(u_{\eps}-\delta)^{+}\chi_{\{|x-z|>r\}}\in H^1,
\qquad
\psi_\eps=(v_{\eps}-\delta)^{+}\chi_{\{|x-z|>r\}}\in H^1,
$$
as test functions for system \rife{eqg}. By multiplying and integrating
over $\Rn$, we obtain
\begin{align*}
&\int_{\R^{3}\diff B(z,r)}\left(\eps^{2}|\grad
(u_{\eps}-\delta)^{+}|^{2}+V(x)u_{\eps}(u_{\eps}-\delta)^{+}
-G_{u}(x,u_{\eps},v_{\eps})(u_{\eps}-\delta)^{+}\right)\\
&+\int_{\R^{3}\diff B(z,r)}\left(\eps^{2}|\grad
(v_{\eps}-\delta)^{+}|^{2}+W(x)v_{\eps}(v_{\eps}-\delta)^{+}
-G_{v}(x,u_{\eps},v_{\eps})(v_{\eps}-\delta)^{+}\right)=0.
\end{align*}
Note that, since we can write
\bdm
G_{u}(x,u_{\eps},v_{\eps})= \begin{cases}
\dfrac{G_{u}(x,u_{\eps},v_{\eps})}{u_{\eps}}\left[(u_{\eps}-\delta)+\delta\right]
&\text{if $u_\eps(x)>0$}, \\
0 &\text{if $u_\eps(x)=0$},
\end{cases}
\edm
and
\bdm
G_{v}(x,u_{\eps},v_{\eps})= \begin{cases}
\dfrac{G_{v}(x,u_{\eps},v_{\eps})}{v_{\eps}}\left[(v_{\eps}-\delta)+\delta\right]
&\text{if $v_\eps(x)>0$}, \\
0 &\text{if $v_\eps(x)=0$},
\end{cases}
\edm
the preceding identity turns into
\begin{align*}
& \int_{\R^{3}\diff B(z)}\big(\eps^{2}|\grad (u_{\eps}-\delta)^{+}|^{2}
+\Upsilon_\eps(x)|(u_{\eps}-\delta)^{+}|^{2}+\Upsilon_\eps(x)\delta(u_{\eps}-\delta)^{+}\big) \\
&+\int_{\R^{3}\diff B(z)}\big(\eps^{2}|\grad (v_{\eps}-\delta)^{+}|^{2}
+\Lambda_\eps(x)|(v_{\eps}-\delta)^{+}|^{2}+\Lambda_\eps(x)\delta(v_{\eps}-\delta)^{+}\big)=0,
\end{align*}
where we have set
\bdm
\Upsilon_\eps(x)= V(x)-\gamma\dfrac{u_{\eps}^{2}(x)+b v_{\eps}^{2}(x)}
{\sqrt{u_{\eps}^{4}(x)+2bu_{\eps}^{2}(x)v_{\eps}^{2}(x)+v_{\eps}^{4}(x)}}
\edm
and
\bdm
\Lambda_\eps(x)= W(x)-\gamma\dfrac{v_{\eps}^{2}(x)+b u_{\eps}^{2}(x)}
{\sqrt{u_{\eps}^{4}(x)+2bu_{\eps}^{2}(x)v_{\eps}^{2}(x)+v_{\eps}^{4}(x)}}.
\edm
By virtue of \rife{kg}, it is easy to show that $\Upsilon_\eps(x)\geq 2\alpha/3$
and $\Lambda_\eps(x)\geq 2\alpha/3$ for all $x$ with $u_\eps(x)>0$ or
$v_\eps(x)>0$, which implies that $(u_{\eps}(x)-\delta)^{+}=0$ and
$(u_{\eps}(x)-\delta)^{+}=0$ for every $x\not\in B(z,r)$ and every
$0<\eps<\eps_\delta$, namely the assertion.
\end{proof}

When proving Theorem \ref{concentrazione} we will use Theorem 2.9
in \cite{mmp} which gives a ne\-ces\-sary condition for the existence of vector
ground state (that is a ground state $(u,v)$ with $u>0$ and $v>0$) for an
autonomous system. Here, for the reader convenience, we briefly sketch the proof in the presence of a two different constant potentials.

\begin{proposition}\label{leastp}
Let $\kappa_1,\kappa_2>0$ and  $(u,v)\in\hh$ be a least energy solution of the system
\beq\label{lili}
\begin{cases}
-\Delta u + \kappa_{1} u = u^3+bv^2u & \text{in $\R^3$}, \\
\noalign{\vskip2pt}
-\Delta v + \kappa_{2} v =v^3+bu^2v & \text{in $\R^3$}.
\end{cases}
\eeq
Let $b_0$ and $b_1$ be defined by
\be\label{defb01}
b_0=\max\left\{\sqrt[4]{\frac{\kappa_1}{\kappa_2}},
\sqrt[4]{\frac{\kappa_2}{\kappa_1}}\right\}\quad b_1=\max
\left\{h\left(\sqrt{\frac{\kappa_1}{\kappa_2}}\right),
h\left(\sqrt{\frac{\kappa_2}{\kappa_1}}\right)\right\},
\ee
where $h(s)$ is defined in \eqref{defh1}.
 Then the following facts holds:
\vskip4pt
\noindent
{\rm (a)} if $b<b_0$ then either $u\equiv 0$ and $v\not\equiv 0$ or
$u\not\equiv 0$ and $v\equiv 0$.
\vskip2pt
\noindent
{\rm (b)} if $b>b_1$ then $u\not\equiv 0$ and $v\not\equiv 0$.
\end{proposition}
\begin{proof}
Suppose that $(u,v)$ is a vector ground state of \eqref{lili} and assume,
without loss of generality, that $0<\kappa_{2}\leq\kappa_{1}$. Consider the
functions
\bdm
\overline{u}(x)=\dfrac{1}{\sqrt{k_{1}}}u\left( \dfrac{x}{\sqrt{k_{1}}} \right)\quad
\overline{v}(x)=\dfrac{1}{\sqrt{k_{1}}}v\left( \dfrac{x}{\sqrt{k_{1}}} \right),
\edm
the above system becomes
\bdm
\begin{cases}
-\Delta\ovu +\ovu =\ovu^3+b\ovv^2\ovu & \text{in $\R^3$}, \\
\noalign{\vskip2pt}
-\Delta \ovv + \om^{2} \ovv =\ovv^3+b\ovu^2\ovv & \text{in $\R^3$},
\end{cases}
\edm
where we set $\om^{2}=\kappa_{2}/\kappa_{1}\leq1$. Then, conclusion (a) follows by applying
\cite[Theorem 2.9]{mmp}, whereas conclusion (b) can be obtained by arguing as
in the proofs of \cite[Theorems 2.3, 2.8]{mmp} (see Remark 2.11 therein).
\end{proof}

\vskip2pt
\begin{proof}[{\it Proof of Theorem \ref{concentrazione}.}]
By virtue of Proposition \ref{maximo}, taking into account the
definition of $G$, the pair $(u_{\eps},v_{\eps})\neq(0,0)$ turns out to be a
solution of \eqref{problema}. From elliptic regularity theory it follows that
$u_\eps,\,v_\eps$ are nonnegative $C^2$ functions.  Let $\xi_\eps$ a local
maximum point of the function $u_{\eps}(x)+v_{\eps}(x)$, then
\begin{align*}
0\leq -\Delta (u_{\eps}+v_{\eps})(\xi_\eps)=& -V(\xi_\eps)u_{\eps}(\xi_\eps)-W(\xi_\eps)
v_{\eps}(\xi_\eps)    \\
&+\left(u^{2}_{\eps}(\xi_\eps)+bv^{2}_{\eps}(\xi_\eps)\right)
u_{\eps}(\xi_\eps)+\left(v^{2}_{\eps}(\xi_\eps)+bu^{2}_{\eps}(\xi_\eps)\right)v_{\eps}(\xi_\eps).
\end{align*}
Using \rife{positiveVW}, there exists a positive radius $\sigma$,
independent on $\eps$,  such that
\beq\label{stimainf}
(u_{\eps}+v_{\eps})(\xi_\eps)\geq\sigma.
\eeq
Let us first  prove \eqref{conveVW} of  conclusion (i) in
Theorem \ref{concentrazione} arguing by contradiction.
More precisely, consider $\eps_n\to 0$ and $x_n\in B(z,r)$
 a local maximum point of $u_{\eps_n}+v_{\eps_n}$.
Let $x_n\to x^*\in\overline{B}(z,r)$, and assume that
both $V(x^*)>V_0$ and $W(x^*)>W_0$. Then, we can consider the
sequences $\phi_n(x)=u_{\eps_n}(x_n+\eps_n x)$,
$\psi_n(x)=v_{\eps_n}(x_n+\eps_n x)$ and the limit $(\phi,\psi)$, critical point of the
limit functional $I_{x^*}$. First, note that $(\phi,\psi)\not=(0,0)$
thanks to \eqref{stimainf}; moreover,  by virtue of the inequalities $V(x^*)>V_0$ and
$W(x^*)>W_0$, the critical level $I_{x^{*}}(\phi,\psi)$ can be compared with
$\Sigma(z)$, yielding again a contradiction.
Then, in order to prove conclusion (i) of Theorem \ref{concentrazione}, it is
only left to show the uniqueness of the maximum point of  the function
$u_{\eps}+v_{\eps}$. Assume by contradiction that there exist
a sequence $\{\eps_{n}\}$ converging to zero and  two local
maxima $x^{1}_{n}$, $x^{2}_{n}$ $\in \overline{B}(z,r)$, which both satisfy
\eqref{stimainf}. We consider the sequences
$$
\phi_{n}(x)=u_{\eps_{n}}(x^{1}_{n}+\eps_{n}x)\quad\,\, \text{and}\,\, \quad
\psi_{n}(x)=v_{\eps_{n}}(x^{1}_{n}+\eps_{n}x).
$$
Arguing as before, we show that the couple $(\phi_{n},\psi_{n})$ converges in
the $C^{2}$ sense over compacts to a solution $(\phi,\psi)$ of \eqref{limit-z}
with $z=x_1$ and  $V(x^{1})=V_0$ and $W(x^{1})=W_0$.
From \eqref{stimainf} we get that $(\phi,\psi)\neq (0,0)$ and from
\cite{busi} we deduce that $(\phi,\psi)$ are nonnegative, radially
symmetric functions. Then the sum $\phi+\psi$ has a local non-degenerate
maximum point, which, up to translations,  is located
in the origin. This facts and the $C^{2}$
convergence of $\phi_{n}+\psi_{n}$ imply that $x_n=(x^{2}_{n}-x^{1}_{n})/\eps_{n}\to
\infty$.  Then we can argue as in the proof  of \eqref{claim} to get a contradiction.
Indeed, we consider the function
$$
h_n=\frac12\left[|\nabla \phi_n|^2+|\nabla \psi_n|^2+V(x^1_n+\eps_n x)\phi_n^2+
W(x^1_n+\eps_n x)\psi_n^2\right]-F(\phi_n,\psi_n).
$$
For every $\delta$ we can choose $R>0$ and $n_{0}$
sufficiently large such that $B_R\cap B_R(x_n)=\emptyset$ for every $n\geq n_{0}$
and
\be\label{dis1}
\lim_{n\to\infty}\int_{B_{R}(0)}h_n  \geq I_{x^1}(\phi,\psi)-\delta.
\ee
Moreover,
\begin{align*}
\lim_{n\to\infty}\int_{B_R(x_n)}\!\!\!\! h_n\!\! =\!\! \frac12\lim_{n\to\infty}
&\left\{\int_{B_R}\!\!\!\left[|\nabla \ovphi_n |^2+
|\nabla \ovpsi_n|^2+V(x^2_{n}+\eps_{n}x_{n})\ovphi_n^2\right.\right. \\
&\left.\left.\,\,\,\,+W(x^2_{n}+\eps_{n}x_{n})
\ovpsi_n^2\right]-\int_{B_R}F(\ovphi_n,\ovpsi_n)\right\}
\end{align*}
where we put $\ovphi_{n}(y)=\phi_{n}(y+x_{n})$, $\ovpsi_{n}(y)=\psi_{n}(y+x_{n})$.
As $V(x^{1})=V(x^{2})=V_0$ and $W(x^{1})=W(x^{2})=W_0$, we get
\be
\label{dis2}
\lim_{n\to\infty}\int_{B_R(x_n)}h_n \geq
I_{x^{2}}(\ovphi,\ovpsi)-\delta=I_{x^{1}}(\phi,\psi)-\delta.
\ee
Then, arguing as in the proof of \eqref{claim} we get
$$
\liminf_{n\to\infty}J_{n}(\phi_{n},\psi_{n})\geq 2 \Sigma(x^{1})=2\Sigma(z),
$$
which is in contradiction with \eqref{energia}.
\vskip3pt
\noindent
In order to prove the exponential decay, notice that, by Proposition \ref{maximo},
$u_{\eps}$ and $v_{\eps}$ decay to zero at infinity, uniformly with respect to $\eps$. Hence we find
$\rho>0$, $\Theta\in(0,\sqrt\alpha)$ and $\eps_0>0$ such that
$u_\eps^2+bv_\eps^2 \leq \alpha-\Theta^2$ and
$v_\eps^2+bu_\eps^2 \leq \alpha-\Theta^2$,
for all $|x-x_\eps|>\eps\rho$ and $0<\eps<\eps_0$. Let us set
$$
\xi_\rho(x)=M_\rho e^{-\Theta(|\frac {x-x_\eps}{\eps}|-\rho)},\qquad
M_\rho=\sup_{(0,\eps_0)}\max_{|x|=\rho} (u_\eps+v_\eps),
$$
and introduce the set ${\mathcal A}=\bigcup_{R>\rho}D_R,$
where, for any $R>\rho$,
$$
\quad D_R=\big\{\rho<|x|<R:\,\,
u_\eps(x)+v_\eps(x)>\xi_\rho(x)\,\,\,\,
\text{for some $\eps\in(0,\eps_0)$}\big\}.
$$
Assume by contradiction that ${\mathcal A}\not=\emptyset$.
Then there exist $R_*>\rho$ and $\eps_*\in(0,\eps_0)$ with
\begin{align*}
\eps^2\Delta(\xi_\rho-u_{\eps_*}-v_{\eps_*}) &
\leq\left[\Theta^2-\frac{2\eps\Theta}{|x-x_\eps|}\right]\xi_\rho
-\Theta^2 u_{\eps_*}-\Theta^2 v_{\eps_*} \\
\noalign{\vskip2pt}
& \leq\Theta^2 (\xi_\rho-u_{\eps_*}-v_{\eps_*})<0,
\qquad\text{in $D_R$ for all $R\geq R_*$}.
\end{align*}
Hence, by the maximum principle, we get
$$
\xi_\rho-u_{\eps_*}-v_{\eps_*} \geq
\min\Big\{\min_{|x|=\rho}(\xi_\rho-u_{\eps_*} -v_{\eps_*}),
\min_{|x|=R}(\xi_\rho-u_{\eps_*} -v_{\eps_*})\Big\},
$$
in $D_R$ for all $R\geq R_*$. Letting $R\to\infty$ and
recalling the definition of $\xi_\rho$ yields
$$
\xi_\rho-u_{\eps_*}-v_{\eps_*}\geq \min
\Big\{\min_{|x|=\rho}(\xi_\rho-u_{\eps_*}-v_{\eps_*}),0\Big\}\geq 0,
\quad\text{in $\bigcup_{R\geq R_*}D_R$}.
$$
In turn, $u_{\eps_*}(x)+v_{\eps_*}(x)\leq \xi_\rho(x)$ for
all $x$ in $\cup_{R\geq R_*}D_R$, which yields a contradiction.
Whence ${\mathcal A}=\emptyset$, and the desired exponential decay follows.
\vskip2pt

Now we prove conclusion (ii) of Theorem \ref{concentrazione}. Once again, let us set
$(\phi_{\eps},\psi_{\eps})=(u_{\eps}(x_{\eps}+\eps x), v_{\eps}(x_{\eps}+\eps x))$.
Note that \rife{stimato} gives us $\|(\phi_{\eps},\psi_{\eps})\|_{\H} \leq C$
and the pair $(\phi_{\eps},\psi_{\eps})$ solves
\bdm
\begin{cases}
-\Delta\phi_{\eps}+V(x_{\eps}+\eps x)\phi_{\eps}
=\phi_{\eps}^3+b\psi_{\eps}^2\phi_{\eps} & \text{in $\R^3$} \\
\noalign{\vskip4pt}
-\Delta\psi_{\eps}+W(x_{\eps}+\eps x)\psi_{\eps}
=\psi_{\eps}^3+b\phi_{\eps}^2\psi_{\eps} & \text{in $\R^3$}.
\end{cases}
\edm
From the conclusion (i) we have that $x_\eps$ converges to $p$, with
$V(p)=V_0$ and $W(p)=W_0$, and $(\phi_\eps,\psi_\eps)$ converges to
$(\phi,\psi)$, least energy solution of \eqref{lili} with
$\kappa_{1}=V_{0}$ and $\kappa_{2}=W_{0}$.
Then, if $b<b_0$, in the light of Proposition \ref{leastp}, either
$\phi\equiv 0$ or $\psi\equiv 0$.
Since  $\phi_\eps$ and $\psi_\eps$ converge uniformly over compacts, we have
that either $u_\eps(x_\eps)=\phi_\eps(0)\to 0$ or $v_\eps(x_\eps)=\psi_\eps(0)\to 0$.
Similarly, if $b>b_1$, in the light of Proposition \ref{leastp} $\phi\neq 0$ and $\psi\neq 0$,
and the assertion follows.
\end{proof}

\begin{remark}\label{survivor}
In the previous theorem we have proved that the least energy solution
$(u_\eps,v_\eps)$ converges (up to scalings) to a least energy (by
\eqref{energia}) solution $(\phi,\psi)$ of
\be\label{sistem}
\begin{cases}
-\Delta \phi+V_0\phi=\phi^3+b\psi^2\phi, & \\
-\Delta \psi+W_0\psi=\psi^3+b\phi^2\psi. &
\end{cases}
\ee
Moreover, for $b<b_0$, one between $\phi,\psi$ is necessarily zero;
so that $(\phi,\psi)$ is actually either $(\phi,0)$ or $(0,\psi)$, with $\phi$
(respectively $\psi$) the unique least energy solution of
$-\Delta \phi+V_0\phi=\phi^3$ (respectively $-\Delta \psi+W_0\psi=\psi^3$).
Then, if $V_0<W_0$, the least scalar energy solution of \eqref{sistem}  is
$(\phi,0)$, yielding $v_\eps(x_\eps)\to0$. Otherwise, if $W_0<V_0$,
$u_\eps(x_\eps)\to0$.
\end{remark}

\vskip4pt
\begin{proof}[{\it Proof of Theorem \ref{concentrazioneG}.}]
It suffices to run through the various steps of the proof of
Theorem \ref{concentrazione} up to formula \eqref{diseq}.
Now, in order to obtain \eqref{cruc-step} we can use hypothesis
\eqref{minimoS} instead of \eqref{minimoV}, \eqref{minimoW} to get directly
\begin{equation*}
\Sigma(z)<\Sigma(x_0)\leq J_{x_0}(\phi,\psi)\leq
\liminf_{n\to\infty}J_{n}(\phi_{n},\psi_{n})\leq \Sigma(z),
\end{equation*}
as $x_0\in\partial B(z,r)$ and $z\in B(z,r)$, yielding the desired contradiction and
thus eventually proving Proposition \ref{maximo}. If $x_\eps$ is the
sequence of maximum points, there holds $\Sigma(x_\eps)\to \Sigma_0$,
otherwise one would get a contradiction similar to the one above.
The dichotomy and the exponential decay can be proved exactly
as we have done in the proof of Theorem \ref{concentrazione}.
\end{proof}


\section{Proof of Theorem \ref{necessthm}}
\label{nec-ps}

In this section we will prove  Theorem \ref{necessthm}. To this aim,
the following preliminary lemma will be useful.

\begin{lemma}
\label{criticiVW}
Assume that $V,\,W\in C^1(\Rn)$ satisfy \eqref{growthVW}.
If $z\in\mathcal{E}$, then
\beq\label{combin}
\gamma_1(z)\nabla V(z)+\gamma_2(z)\nabla W(z)=0,
\ee
for some $\gamma_1(z)\geq 0$, $\gamma_2(z)\geq 0$, one of them being nontrivial.
\end{lemma}

\begin{proof}
Let $z\in\mathcal{E}$,  $\eps_n$ a sequence converging to zero and
$(u_{\eps_n},v_{\eps_n})$ solution of problem \eqref{problema} that satisfies
the properties in  Definition \ref{defe}. Let us define
$\vfi_n(x)=u_{\eps_n}(z+\eps_n x)$, $\psi_n(x)=v_{\eps_n}(z+\eps_n x)$
and the lagrangian
${\mathcal L}: \Rn \times \R \times \R \times \Rn \times \Rn \to\R$ defined as
\bdm
{\mathcal L}(x,s_1,s_2,\xi_1,\xi_2)=
\frac{|\xi_1|^2+|\xi_2|^2}{2}+V(z+\eps_n x)\frac{s_1^2}{2}
+W(z+\eps_n x) \frac{s_2^2}{2}-\frac{s_1^4+2bs_1^2s_2^2+s_2^4}{4}.
\edm
By the Pucci-Serrin identity for systems \cite[see \textsection 5]{puc-ser}, we have
\begin{align*}
& \sum\limits^3_{i,\ell=1}
\irn \partial_i{\boldsymbol h}^\ell  \partial_i \psi_n \, \partial_\ell \psi_n
+\sum\limits^3_{i,\ell=1} \irn\partial_i {\boldsymbol h}^\ell  \partial_i \vfi_n \, \partial_\ell \vfi_n \\
&=\irn {\rm div} {\boldsymbol h}\, {\mathcal L}(x,\vfi_n,\psi_n,\n \vfi_n,\n \psi_n) \\
&+\frac 12\irn \eps_n{\boldsymbol h} \cdot  [\nabla V(z+\eps_n x)\varphi_n^2+
\nabla W(z+\eps_n x)\psi_n^2],
\end{align*}
for all ${\boldsymbol h} \in C^1_{\rm c}\left(\Rn, \Rn \right)$.
Let us choose, for any $\lambda>0$,
$$
{\boldsymbol h}_j:\R^3\to\R^3,\qquad
{\boldsymbol h}_j^\ell(x)=
\begin{cases}
\Upsilon(\lambda x) & \text{if $j=\ell$}, \\
0 & \text{if $j\neq \ell$},
\end{cases}\qquad \ell=1,2,3,
$$
$\Upsilon\in C^1_{{\rm c}} (\Rn)$,
$\Upsilon(x)=1$ if $|x|\leq 1$ and $\Upsilon(x)=0$
if $|x|\geq 2$. Then, for $j=1,\dots,3$,
\begin{align*}
& \sum\limits^3_{i=1}
\irn \lambda \partial_i \Upsilon(\lambda x) \partial_i \psi_n \, \partial_j \psi_n
+\sum\limits^3_{i=1}\irn \lambda \partial_i \Upsilon(\lambda x) \partial_i \vfi_n \, \partial_j \vfi_n \\
&=\irn \lambda \partial_j \Upsilon(\lambda x) {\mathcal L}(x,\vfi_n,\psi_n,\n \vfi_n,\n \psi_n) \\
&+\frac 12\irn \eps_n \Upsilon(\lambda x)[\partial_j V(z+\eps_n x)\varphi_n^2+
\partial_j W(z+\eps_n x)\psi_n^2].
\end{align*}
By the arbitrariness of $\lambda>0$, letting
$\lambda \to 0$ and keeping $j$ fixed, we obtain
\[
\irn [\partial_j V(z+\eps_n x)\varphi_n^2+
\partial_j W(z+\eps_n x)\psi_n^2]=0
\qquad j=1,2,3.
\]
By assumption \eqref{growthVW}, there exists a positive constant $\beta_1$
such that, for all $x\in\R^3$ and $j\geq 1$, we get
$|\nabla V(z+\eps_n x)|\leq \beta_1 e^{\gamma\eps_n |x|}$
and $|\nabla W(z+\eps_n x)|\leq \beta_1 e^{\gamma\eps_n |x|}$,
so that, invoking the uniform exponential decay of $\varphi_n$ and $\psi_n$, letting
$n\to \infty$ in the above identity, there holds
\begin{equation}\label{part-fin}
\irn (\partial_j V(z)\varphi_z^2+\partial_j W(z)\psi_z^2)=0,
\qquad j=1,2,3,
\end{equation}
where $(\varphi_{z},\psi_{z})\neq(0,0)$ is a least energy solution of \rife{limit-z}.

Therefore \eqref{combin} holds with
$\gamma_1(z)=\|\varphi_z\|_{2}^2$ and $\gamma_2(z)=\|\psi_z\|_{2}^2$.
\end{proof}

\begin{proof}[{\it Proof of Theorem \ref{necessthm}.}]
First, we will show that  $\Sigma$ is a continuous function.
Recall from \cite[Lemma 3.1]{mmp}
that, for every $\xi\in \Rn$ and  $w\in H^1\times H^1$ with $w\not=(0,0)$,  there exists
a unique $\theta(w,\xi)>0$ such that $\theta(w,\xi)w\in {\mathcal N}_\xi$
(defined in \eqref{defnehari}); the map $\{w\mapsto \theta(w,\xi)\}$
is continuous and $\{w\mapsto \theta(w,\xi)w \}$ is a homeomorphism
of the unit sphere of $H^1\times H^1$ on ${\mathcal N}_\xi$.
In order to prove that $\Sigma$ defined in \eqref{defsigma} is
continuous, let us first consider  the potentials $V(x),\,W(x)$
as positive constants $V,\,W\in \R^+$. Following the line of \cite{rabinowitz},
we first show the continuity of the map
$(V,W)\to c(V,W)$, where $c(V,W)$ is the mountain pass level of the functional
 $I_{V,W}:H^1\times H^1\to\R$ defined by
\begin{equation*}
I_{V,W} (u,v)= \frac{1}{2}\int_{\R^3}|\nabla u|^2+|\nabla v|^2
+Vu^2+Wv^2- \int_{\R^3} F(u,v).
\end{equation*}
The following equalities hold (see Lemma 3.1 in \cite{mmp})
\begin{equation}\label{ceqn}
c(V,W)=\inf_{H^1\times H^1\setminus(0,0)}\max_{t\geq0}I_{V,W}(tu,tv)=
\inf_{{\mathcal N}_{V,W}} I_{V,W}
\end{equation}
where ${\mathcal N}_{V,W}$ is the Nehari manifold associated to $I_{V,W}$.
Note that \eqref{ceqn} implies that proving the continuity of the map $c(V,W)$
is equivalent to show the continuity of the map $(V,W)\mapsto \Sigma(V,W)$.
Let us first show that
\begin{equation}
\label{claimc} \lim_{\eta \to 0} c(V+\eta,W+\eta) = c(V,W).
\end{equation}
It is readily seen that the following monotonicity property holds
\begin{equation}
\label{dismp} V_1>V_2,\,\, W_1>W_2\,\,\,\,\Longrightarrow\,\,\,\,
c(V_1,W_1) \geq c(V_2,W_2).
\end{equation}
By virtue
of~\eqref{dismp}, we get
\begin{equation}\label{cmeno}
\lim_{\eta \to 0^-} c(V+\eta,W+\eta):= c^- \leq c(V,W).
\end{equation}
Let $\eta_h \to 0^-$ and $\delta_h \to 0^+$ as
$h\to\infty$. By the definition of $c(V+\eta,W+\eta)$ and \eqref{ceqn},
and since the map $\theta$
induces an homeomorphism of the unit sphere of $H^1\times H^1$
on ${\mathcal N}_{V+\eta_h,W+\eta_h}$, there exists
$(u_h,v_h)\in H^1\times H^1$, such that
\begin{align}
\label{norma}
\|\nabla u_h\|_2^2+\|\nabla v_h\|_2^2 &+\|u_h\|_2^2+\|v_h\|_2^2=1,
\\\label{maxinequal}
\max _{t\geq 0}I_{V+\eta_h,W+\eta_h}(t u_h,t v_h)&\leq
c(V+\eta_h,W+\eta_h)+\delta_h.
\end{align}
We will  first show that $\theta(u_{h},v_{h})$,
 given by
\begin{equation}
\label{formult0} \theta (u_h,v_h)=\sqrt\frac{\|\nabla
u_h\|^2_{2}+\|\nabla v_h\|^2_{2} +V\|u_h\|^2_{2}+W\|v_h\|^2_{2}}
{\|u_h\|^4_{4}+\|v_h\|^4_{4}+2b\|u_hv_h\|^2_{2}},
\end{equation}
remains bounded. We argue by contradiction, therefore, we suppose, in virtue of
\eqref{norma} that
\be\label{zero}
\|u_h\|^4_{4}+\|v_h\|^4_{4}+2b\|u_hv_h\|^2_{2}\to 0.
\ee
From the Ekeland variational principle we obtain that there exists a
sequence $(\xi_{h},z_{h})$ such that
\begin{align}
\label{dista}\|u_h-\xi_h\|_{H^{1}}+\|v_h-z_h\|_{H^{1}} &\leq \sqrt\delta_h,
\\
\nonumber c(V+\eta_h,W+\eta_h)-\delta_h<I_{V+\eta_h,W+\eta_h}(\xi_h,z_h)
&<c(V+\eta_h,W+\eta_h)+\delta_{h},
\\
\nonumber I_{V+\eta_h,W+\eta_h}'(\xi_h,z_h) &\to 0.
\end{align}
From \eqref{dista} and \eqref{zero} it follows that
$$
\|\xi_h\|^4_{4}+\|z_h\|^4_{4}+2b\|\xi_hz_h\|^2_{2}\to 0.
$$
Then
\begin{align*}
0 <c^- &=\lim_{h\to \infty}\left\{ I_{V+\eta_h,W+\eta_h}(\xi_h,z_h)-
\frac{1}{2}\langle I_{V+\eta_h,W+\eta_h}'(\xi_h,z_h),(\xi_h,z_h)\rangle\right\}
\\
&=\frac14\lim_{h\to \infty}
\left\{\|w_h\|^4_{4}+\|z_h\|^4_{4}+2b\|w_hz_h\|^2_{2}\right\}=0,
\end{align*}
which is an obvious contradiction, proving that $\theta(u_h,v_h)$ remain bounded.
Denoting with $\theta(u,v)=\theta(u,v,V,W)$ and using the
definition we have
 $$
I_{V,W}(\theta(u,v)u,\theta(u,v)v )=\max_{t\geq 0} I_{V,W}(tu,tv).
$$
In virtue of \eqref{ceqn}, \eqref{dismp}, \eqref{cmeno} and \eqref{maxinequal},
it results
\begin{align*}
c(V,W) &
\leq I_{V,W}(\theta(u_h,v_h)u_h,\theta(u_h,v_h)v_h)
\\
&= I_{V+\eta_h,W+\eta_h} (\theta(u_h,v_h)u_h,\theta(u_h,v_h)v_h)
- \frac{\eta_h}2\theta^{2}(u_h,v_h)( \|u_h\|_2^2+\|v_h\|_2^2)
\\
&\leq c(V+\eta_h,W+\eta_h) + \delta_h -\frac{ \eta_h}2 \theta^2(u_h,v_h)
( \|u_h\|_2^2+\|v_h\|_2^2)
\\
&\leq c^- + \delta_h - \frac{\eta_h}2 \theta^2(u_h,v_h)
( \|u_h\|_2^2+\|v_h\|_2^2)
\\
&\leq c(V,W)+ \delta_h - \frac{\eta_h}2 \theta^2(u_h,v_h)
( \|u_h\|_2^2+\|v_h\|_2^2).
\end{align*}
From \eqref{norma} and as $\theta$ is an homeomorphism on the unit sphere,
it follows,  for $h\to\infty$, that  $c(V,W)=c^-$.
In a similar fashion one can prove that
\begin{equation}\label{claimc2}
c(V,W) =\lim_{\eta\to 0^+} c(V+\eta,W+\eta)
\end{equation}
Therefore \eqref{claimc} is proved.
Let now $\{z_h\}$ be a sequence in $\R^3$ such that $z_h\to z$ as
$h\to\infty$. Observe that, given $\eta >0$, for large $h$, we
have
\begin{align*}
V(z)+\eta &\geq V(z)+|V(z_h)-V(z)| \\
&\geq V(z_h) \geq V(z) - |V(z_h)-V(z)|\geq V(z) - \eta,
\end{align*}
and similar relations hold for $W$. From \eqref{claimc} and \eqref{claimc2}
we deduce that $c(V(z)+\eta,W(z)+\eta)$ and $c(V(z)-\eta,W(z)-\eta)$
both converge to $c(V(z),W(z))$, yielding the desired continuity of
$z\mapsto \Sigma(z)$.

Let us show that the function $\Sigma$ defined in \eqref{defsigma} is
locally  Lipschitz continuous. We denote by $\SS(z)$ the set of the nonnegative radial
critical points of $I_{z}$ of least energy.
Let $z\,\xi\in\R^3$ and $(\phi_z,\psi_z)\in\SS(z)$, we denote here
 $\theta (z,\xi)=\theta(\phi_z,\psi_z,\xi)=\theta(\phi_z,\psi_z,V(\xi),W(\xi))$.
Then
\begin{equation*}
\Sigma(\xi)-\Sigma(z)\leq I_\xi(\theta(z,\xi)(\phi_z,\psi_z))-I_z(\phi_z,\psi_z).
\end{equation*}
Defining the function
\beq\label{defh}
h(\xi)=I_\xi(\theta(z,\xi)(\phi_z,\psi_z))
\eeq
and noting that $\theta(z,z)=1$ we obtain
\beq\label{dissigma}
\Sigma(\xi)-\Sigma(z)\leq h(\xi)-h(z).
\eeq
In order to prove that $\Sigma$ is locally Lipschitz, we will use the mean value theorem
applied to the function $h(\xi)$, so that we will show that  $\nabla h$ is bounded.
First observe that, since $\theta (z,\xi) (\phi_z,\psi_z)
\in{\mathcal N}_\xi$ we get that $\theta (z,\xi)$ is given by
\eqref{formult0} with $u_h=\phi_h$, $v_h=\psi_h$ and  $V=V(\xi)$, $W=W(\xi)$.
From the continuity of the critical level in dependence of $V(\xi), \,W(\xi)$
and from the continuity of $\Sigma $ we obtain that the functions
\begin{align*}
 (z,\xi) &\mapsto \|\nabla \phi_z\|^2_{2}+\|\nabla \psi_z\|^2_{2}
+V(\xi)\|\phi_z\|^2_{2}+W(\xi)\|\psi_z\|^2_{2},
\\
z & \mapsto \|\phi_z\|^4_{4}+\|\psi_z\|^4_{4}
+2b\|\phi_z \psi_z\|^2_{2}
\end{align*}
remain bounded and away from zero from below as $z$ and $\xi$ remain bounded,
so that $\theta(z,\xi)$ remains bounded for $(z,\xi)$ bounded.
Moreover,  $\theta(z,\xi)$ is differentiable with respect to the
variable $\xi$ so that also the function $h$ defined in \eqref{defh} is differentiable and
its gradient  is  given by
\begin{align*}
\nabla h(\xi)=&\n_\xi I_\xi (\theta (z,\xi) (\phi_z,\psi_z))
= \frac{\theta (z,\xi)^2}2 \left[\nabla V(\xi)\|\phi_z\|_2^2+\nabla W(\xi)\|\psi_z\|_2^2\right]
\\
& +\theta (z,\xi)\nabla_\xi \theta (z,\xi)\left[
\|\nabla \phi_z|_2^2+\|\nabla \psi_z\|_2^2+V(\xi)\|\phi_z\|_2^2+W(\xi)\|\psi_z\|^2
\right]
\\
&-\theta (z,\xi)^3\nabla_\xi\theta(z,\xi)\left[
\|\phi_z\|_4^4+\|\psi_z\|_4^4+2b\|\phi_z\psi_z\|_2^2)\right],
\end{align*}
so that
\begin{align*}
\nabla h(\xi)
&=\frac{\theta (z,\xi)^2}{2}\left[\nabla V(\xi)\|\phi_z\|^2+\nabla W(\xi)\|\psi_z\|_2^2\right]
\\
&+\frac{\nabla_\xi\theta(z,\xi)}{\theta(z,\xi)^2}I_\xi'
(\theta (z,\xi)\phi_z,\theta (z,\xi)\psi_z)[\theta (z,\xi)\phi_z,\theta (z,\xi)\psi_z].
\end{align*}
Hence, since $(\theta (z,\xi)\phi_z,\theta (z,\xi)\psi_z)\in \Ne_\xi$,  we get
$$
\nabla h(\xi)=\frac{\theta (z,\xi)^2}{2}\left[\nabla V(\xi)\|\phi_z\|^2
+\nabla W(\xi)\|\psi_z\|_2^2\right]
$$
This formula, \eqref{dissigma},  the mean-value theorem
applied to the function $h$ and the local boundedness of $\theta$
imply that $\Sigma$ is locally Lipschitz (in order to get the opposite
inequality, it suffices to switch $z$ with $\xi$).
\vskip4pt
\noindent
Now, let us prove conclusion (a) of Theorem \ref{necessthm}.

Let $z\in\mathcal{E}$ and $(u_{\eps_n},v_{\eps_n})\subset \H$
a sequence of solutions to~\eqref{problema}
that satisfy the properties in Definition \ref{defe}.
Let us consider for all $n\geq 1$ $\eps_n\to 0$ and
the sequences $\vfi_n (x)=u_{\eps_n}(z+\eps_n x)$,
$\psi_n (x)=v_{\eps_n}(z+\eps_n x)$, so that
$\vfi_n(x)+\psi_n(x)\to 0$ as $|x|\to\infty$,
uniformly with respect to $n$ and
$J_{\eps_n}(\vfi_n,\psi_n)\to\Sigma(z)$ as $n\to\infty$.
The sequence $(\vfi_n,\psi_n)$ converges $C^2$ over compacts
to $(\vfi_z,\psi_z)$ a least energy
solution of  \eqref{limit-z}, and $\vfi_z,\,\psi_z$ are radially and
exponentially decaying (see \cite{busi}), that is $(\vfi_z,\psi_z)$ belongs
to $\SS(z)$.
\\
Consider the scalar problems
\begin{equation}
\label{limitV}
\tag{$S^V_z$}
\begin{cases}
-\Delta u + V(z)u = u^3  \quad\text{in $\R^3$,} &\\
\, u>0,\,\, u\in H^1, &\\
\, u(0)=\max\limits_{\R^3}u,
\end{cases}
\end{equation}
\begin{equation}
\label{limitW}
\tag{$S^W_z$}
\begin{cases}
-\Delta v + W(z)v = v^3 \quad\text{in $\R^3$,}& \\
\, v>0,\,\, v\in H^1, &\\
\, v(0)=\max\limits_{\R^3}v.
\end{cases}
\end{equation}
It is known (see \cite{berplions}, \cite{K}) that \eqref{limitV} and \eqref{limitW} have a
unique ground state solution.
Notice that Proposition \ref{leastp}
implies that, if $z\in\mathcal{O}_{b}$, then $(\varphi_{z},\psi_{z})$ has
necessarily one trivial component. So that, the following possibilities may
occur:
\vskip5pt
\noindent
{\bf I.} $z\in{\mathcal O}_b$ and $\vfi_z=0$ and $\psi_z$ is
a nontrivial solution to \eqref{limitW};
\vskip5pt
\noindent
{\bf II.} $z\in{\mathcal O}_b$ and $\psi_z=0$ and $\vfi_z$
is a nontrivial solution to \eqref{limitV};
\vskip5pt
\noindent
{\bf III.} $z\in\mathcal{E}\setminus\mathcal{O}_b=\mathcal{E}_\Sigma$.
\vskip5pt
\noindent
It is readily seen by a simply scaling that, if $\vfi_z\neq 0$ or $\psi_z\neq 0$,
$$
\vfi_z(x)=\sqrt{V(z)}U_0(\sqrt{V(z)}x),\qquad
\psi_z(x)=\sqrt{W(z)}U_0(\sqrt{W(z)}x),
$$
where $U_0$ is the unique solution to $-\Delta u+u=u^3$.
Since $\psi_{n}$ converges uniformly to $\psi_{z}$, which has its global
maximum point in the origin, case I corresponds to $z\in\mathcal{E}_{W}$.
In such a case, in light of \eqref{combin}, there holds $\gamma_1(z)=0$,
$\gamma_2(z)\neq 0$, namely $z\in {\rm Crit}(W)$.
Arguing as above it is possible to show that the
situation of case II implies that $z\in {\mathcal E}_V$ and
$z\in {\rm Crit}(V)$.
Of course ${\mathcal E}_V\cap {\mathcal E}_W\cap\{V\neq W\}=\emptyset$. Indeed, if
$z^*\in{\mathcal E}_V\cap {\mathcal E}_W\cap\{V\neq W\}$ there would exist two sequences
$(u^1_j,v^1_j)$ and $(u^2_j,v^2_j)$ of solutions to \eqref{problema} such that the corresponding
scaled solutions $(\vfi^1_j,\psi^1_j)$ and $(\vfi^2_j,\psi^2_j)$ converge in the $C^2$ sense over
compact sets to $(\varphi^1_{z^*},\psi^1_{z^*})\in\SS(z^*)$ and
$(\varphi^2_{z^*},\psi^2_{z^*})\in \SS(z^*)$ and such that $\vfi^1_j(0)\geq \delta>0$
(since $z^*\in{\mathcal E}_V$) and $\psi^2_j(0)\geq\delta>0$ (since
$z^*\in{\mathcal E}_W$), for every $j$. As a consequence, letting $j\to\infty$,
we get $\varphi^1_{z^*}\neq 0$ and $\psi^2_{z^*}\neq 0$. Now, in light of
Lemma \ref{leastp}, since $z^*\in{\mathcal O}_b$ and $(\varphi^1_{z^*},\psi^1_{z^*})$
and $(\varphi^2_{z^*},\psi^2_{z^*})$ have least energy, we have $\psi^1_{z^*}=0$
and $\varphi^2_{z^*}=0$. Therefore,
$$
\Gamma\sqrt{V(z^*)}=I_{z^*}(\varphi^1_{z^*},0)=\Sigma(z^*)
=I_{z^*}(0,\psi^2_{z^*})=\Gamma\sqrt{W(z^*)},
$$
contradicting that $V(z^*)\neq W(z^*)$.
The previous facts show that
$$
\mathcal{E}\cap {\mathcal O}_b\subseteq \mathcal{E}_V\cup\mathcal{E}_W,
\,\,\text{and}\,\, \mathcal{E}_V\times\mathcal{E}_W
 \subset{\rm Crit}(V)\times {\rm Crit}(W).
$$
Hence, we conclude that
$$
\mathcal{E}=(\mathcal{E}\cap {\mathcal O}_b)\cup (\mathcal{E}\setminus {\mathcal O}_b)
=\mathcal{E}_V\cup\mathcal{E}_W\cup\mathcal{E}_\Sigma,
$$
with $\mathcal{E}_V\times\mathcal{E}_W\subset {\rm Crit}(V)\times {\rm Crit}(W).$
To prove conclusion (a) of Theorem \ref{necessthm} it  is only left to show that
$ \mathcal{E}_\Sigma\subset {\rm Crit}_C(\Sigma)$.
In order to do this we will first prove that the directional derivatives from the left and the
right of $\Sigma$ at every point $z\in\R^3$ along any $\eta\in\R^3$ exist and it holds
\begin{align*}
&\left(\frac{\partial\Sigma}{\partial \eta}\right)^{-}\!\!(z)=
\sup_{(\varphi_z,\psi_z)\in \SS(z)} \frac{\partial I_z}{\partial \eta}(\varphi_z,\psi_z), \\
&\left(\frac{\partial\Sigma}{\partial \eta}\right)^{+}\!\!(z)=
\inf_{(\varphi_z,\psi_z)\in \SS(z)} \frac{\partial I_z}{\partial \eta}(\varphi_z,\psi_z),
\end{align*}
that is, explicitly,
\begin{align}
\label{form-dder}
\left(\frac{\partial\Sigma}{\partial \eta}\right)^{-}\!\!(z)&=
\sup_{(\varphi_z,\psi_z)\in \SS(z)}\frac{1}{2}\left\{
\frac{\partial V}{\partial \eta}(z)\|\varphi_z\|^2_{2}  +
\frac{\partial W}{\partial \eta}(z)\|\psi_z\|^2_{2}\right\}, \\
\left(\frac{\partial\Sigma}{\partial \eta}\right)^{+}\!\!(z)&=
\inf_{(\varphi_z,\psi_z)\in \SS(z)}\frac{1}{2}\left\{
\frac{\partial V}{\partial \eta}(z)\|\varphi_z\|^2_{2} +
\frac{\partial W}{\partial \eta}(z)\|\psi_z\|^2_{2}\right\},\notag
\end{align}
for every $z,\eta\in\R^3$.

Let $\{\mu_j\}\subset\R^3$ be a sequence con\-ver\-ging to $\mu_0$
and let $(u_j,v_j)$ be a cor\-re\-spon\-ding sequence of solutions
of least energy $\Sigma(\mu_j)$. We want to prove that, up to a subsequence,
$u_j\to u_0$ and $v_j\to v_0$, strongly in $H^1$, with $(u_0,v_0)\in \SS(\mu_0)$.
It is straightforward to see that $(u_j,v_j)$ is bounded in $H^1\times H^1$
so that, up to a subsequence, it converges weakly to a pair $(u_0,v_0)$, and
$u_j\to u_0$ and $v_j\to v_0$ locally in the $C^2$-sense,
so that $(u_0,v_0)$ is a solution to the limiting problem with $\mu=\mu_0$.
Moreover, as previously observed, there exists $\delta>0$ such that
$u_0^2(0)+v_0^2(0)\geq \delta$, which entails $u_0\neq 0$ or $v_0\neq 0$.
Observe that, by the continuity of $\S$ and by Fatou's Lemma, we get
$$
\S(\mu_0)=\lim_{j\to\infty}\S(\mu_j)=
\lim_{j\to\infty}I_{\mu_j}(u_j,v_j)\geq I_{\mu_0}(u_0,v_0)\geq \S(\mu_0).
$$
Hence, in particular, it holds $I_{\mu_j}(u_j,v_j)\to I_{\mu_0}(u_0,v_0)=\S(\mu_0)$
as $j\to\infty$, that is
\begin{gather*}
\lim_{j\to\infty}\int_{\R^3}|\nabla u_j|^2+|\nabla v_j|^2+V(\mu_j)u_j^2+W(\mu_j)v_j^2  \\
=\int_{\R^3}|\nabla u_0|^2+|\nabla v_0|^2+V(\mu_0)u_0^2+W(\mu_0)v_0^2.
\end{gather*}
Then we have $(u_j,v_j)\to (u_0,v_0)$ strongly in $H^1\times H^1$.
For any $(\varphi,\psi)\in \SS(z)$, we get
\begin{align*}
\Sigma(z+t\eta)-\Sigma(z)&\leq I_{z+t\eta}(\vartheta(z,z+t\eta)\varphi,
\vartheta(z,z+t\eta)\psi)
-I_z(\varphi,\psi) \\
&=t \n_\xi I_\xi (\vartheta (\xi,z) \varphi,\vartheta (\xi,z) \psi)|_{\xi\in[z,z+t\eta]}.
\end{align*}
Whence, by  the arbitrariness of $(\varphi,\psi)\in \SS(z)$,
\begin{equation*}
\limsup_{t\to 0^+}\frac{\Sigma(z+t\eta)-\Sigma(z)}{t}
\leq \inf_{(\varphi,\psi)\in \SS(z)}\frac 12\Big\{
\nabla V(z)\cdot\eta\|\varphi\|^2_{2} +
\nabla W(z)\cdot\eta\|\psi\|^2_{2}\Big\}.
\end{equation*}
To get the opposite inequality, take $(\vfi,\psi)\in\mathbb{S}(z+t\eta)$.
It holds
\begin{align*}
\Sigma(z+t\eta)-\Sigma(z)&\geq I_{z+t\eta}(\varphi,\psi)
-I_z(\theta(z+t\eta,z)\varphi,\theta(z+t\eta,z)\psi)
\\
&=t \n_\xi I_\xi (\theta (\xi,z+t\eta) \varphi,\theta (\xi,z+t\eta) \psi)|_{\xi\in[z,z+t\eta]}.
\end{align*}
Using the continuity of $\theta$ and the convergence of $(\varphi,\psi)$ to an
element of $\mathbb{S}(z)$, we obtain
\begin{equation*}
\liminf_{t\to 0^+}\frac{\Sigma(z+t\eta)-\Sigma(z)}{t}
\geq \inf_{(\varphi,\psi)\in \SS(z)}\frac 12\Big\{
\nabla V(z)\cdot\eta\|\varphi\|^2_{2}  +
\nabla W(z)\cdot\eta\|\psi\|^2_{2}\Big\},
\end{equation*}
proving the opposite inequality, so that the desired formula
for the right derivative of $\Sigma$ follows.
A similar argument provides the corresponding formula for the left derivative.
\\
Assume now that $z\in\mathcal{E}\setminus\mathcal{O}_b=\mathcal{E}_\Sigma$.
Notice that, beside \eqref{part-fin}, for all $\eta \in \Rn$, it holds
\[
\irn\frac{\partial V}{\partial \eta}(z)\vfi_z^2
+\frac{\partial W}{\partial \eta}(z)\psi_z^2=0.
\]
Hence, since $(\varphi_z,\psi_z)\in\SS(z)$, by formula \eqref{form-dder} we have
\begin{equation*}
\left(\frac{\partial\Sigma}{\partial \eta}\right)^{+}\!\!(z)\leq 0
\end{equation*}
Then, by the definition of $(-\Sigma)^0(z;\eta)$, we get
\begin{equation*}
(-\Sigma)^0(z;\eta)\geq\left(\frac{\partial (-\Sigma)}
{\partial \eta}\right)^{+}\!\!(z)\geq 0,\quad
\text{for every $\eta\in\R^3$}.
\end{equation*}
In turn $0\in \partial_C(-\Sigma)(z)$ and, since
$\partial_C (-\Sigma)(z)=-\partial_C \Sigma(z)$ (see \cite{clarke}), we
obtain $z\in{\rm Crit}_C(\Sigma)$, which concludes the proof of (a).
\\
If $V$ and $W$ are also bounded from above, by choosing
$b_0^\infty$ and $b_1^\infty$ as in \eqref{bstar1}-\eqref{bstar2} we get
${\mathcal O}_b=\R^3$ for all $b\leq b_0^\infty$ (as $b_0^\infty\leq b_z$ for every $z$),
and ${\mathcal O}_b=\emptyset$ for all $b>b_1^\infty$ (as $b_1^\infty\geq b_z$ for every $z$),
thus immediately proving assertion (b). Finally, if $b>b_2^\infty$ the last assertion
of the theorem follows immediately from Proposition \ref{leastp}.
\end{proof}

\end{document}